\documentclass[12pt,reqno]{amsart}
\usepackage{amsfonts}
\usepackage{a4wide}

\numberwithin{equation}{section}

\newtheorem{theorem}{Theorem}[section]
\newtheorem{proposition}[theorem]{Proposition}
\newtheorem{corollary}[theorem]{Corollary}
\newtheorem{lemma}[theorem]{Lemma}

\theoremstyle{definition}

\newtheorem{definition}[theorem]{Definition}

\theoremstyle{remark}
\newtheorem{remark}[theorem]{Remark}

\newcommand{\be}{\begin{equation}}
\newcommand{\ee}{\end{equation}}

\begin{document}

\title[ scattering for fNLS
 ]
{  Sharp Criteria of Scattering for the Fractional NLS
  }

 \author{Qing Guo and Shihui Zhu
}

\address{College of Science, Minzu University of China, Beijing 100081, China}
\email{guoqing0117@163.com}

\address{Department of Mathematics,
Sichuan Normal University\\ Chengdu, Sichuan 610066, China}
\email{ shihuizhumath@163.com}

\begin{abstract}
In this paper, the sharp threshold of  scattering for the fractional nonlinear Schr\"{o}dinger equation in the $L^2$-supercritical  case
is obtained, i.e.,  if  $1+\frac{4s}{N}<p<1+\frac{4s}{N-2s}$,  and $$ M[u_{0}]^{\frac{s-s_c}{s_c}}E[u_{0}]<M[Q]^{\frac{s-s_c}{s_c}}E[Q],
\ M[u_{0}]^{\frac{s-s_c}{s_c}}\|
u_{0}\|^2_{\dot H^s}<M[Q]^{\frac{s-s_c}{s_c}}\| Q\|^2_{\dot H^s}$$ then the solution $u(t)$ is globally well-posed and scatters. This condition is sharp in the sense that if $1+\frac{4s}{N}<p<1+\frac{4s}{N-2s}$ and $$ M[u_{0}]^{\frac{s-s_c}{s_c}}E[u_{0}]<M[Q]^{\frac{s-s_c}{s_c}}E[Q],
\ M[u_{0}]^{\frac{s-s_c}{s_c}}\|
u_{0}\|^2_{\dot H^s}>M[Q]^{\frac{s-s_c}{s_c}}\| Q\|^2_{\dot H^s},$$ then the corresponding solution $u(t)$ blows up in finite time, according to Boulenger, Himmelsbach, and Lenzmann's results in 
 \cite{b-h-ljfa16}.
\end{abstract}

\maketitle

\section{Introduction}

From expanding the
Feynman path integral from the Brownian-like to the L\'evy-like quantum mechanical paths, Laskin  in \cite{Laskin2000,Laskin2002}  established the fractional  Schr\"{o}dinger equations  from the viewpoint of  Physics, which  have physical applications in the energy spectrum for a hydrogen-like atom-fractional Bohr atom. The studying of the  fractional nonlinear Schr\"{o}dinger equations (fractional NLS, for short) attacking more and more Mathematical researchers (see\cite{BaoDong2011,b-h-ljfa16,ChoHwangHajaiejOzawa2012,ElgartSchlein2007,Frank-Lenzmann2013,FrohlichLenzmann2007,Guohuo2011,Hajaiej2013,Ionescu-Pusateri2014,LenzmannLewin2010,sun-zhengarxiv}). In the present paper,
we   investigate  the  following  Cauchy problem of  the $L^2$-supercritical fractional NLS.
\begin{equation}\label{eq1} iu_t-(-\triangle)^s
u+|u|^{p-1}u=0, \ \ \  \end{equation}
\begin{equation} \label{1.2} u(0,x)=u_{0}\in H^s,
\end{equation}
where $0<s<1$   and the fractional operator $(-\triangle)^{s}$ is defined by
\[(-\triangle)^{s}u=\frac{1}{(2\pi)^{\frac N2}}\int  e^{ix\cdot\xi}|\xi|^{2s}\widehat{u}(\xi)d\xi=\mathcal{F}^{-1}[|\xi|^{2s}\mathcal{F}[u](\xi)],\]   where $\mathcal{F}$ and
$\mathcal{F}^{-1}$ are the Fourier transform and the Fourier inverse
transform in $\mathbb{R}^N$, respectively.  $u=u(t,x)$: $\mathbb{R}\times \mathbb{R}^N \to \mathbb{C}$ is the wave function.
The power exponent $1+\frac{4s}{N}<p<1+\frac{4s}{N-2s}$ (when $N\leq  2s$, $1+\frac{4s}{N}<p<\infty$).

When $1+\frac{4s}{N}<p<1+\frac{4s}{N-2s}$  for $N>2s$, and $1+\frac{4s}{N}<p<\infty$ for $N\leq  2s$,
Eq. (\ref{eq1}) is the   $L^2$-supercritical  fractional NLS due to the scaling invariance.
Indeed,  if $u(t,x)$ is  a solution of Eq.(\ref{eq1}), then $u^{\lambda}(t,x)=\lambda^{\frac{2s}{p-1}}u(\lambda^{2s}t,\lambda x)$ is also a solution of Eq.(\ref{eq1}). Then, we see the following invariant norms.
 \begin{itemize}
\item [(1)] $\|u^{\lambda}\|_{L^{p_c}}= \|u\|_{L^{p_c}}$, where $p_c=\frac{N(p-1)}{2s}$. We remark that  $p_c>2$ when $
p-1>\frac{4s}N$, and then Eq. (\ref{eq1}) is called the $L^2$-supercritical NLS.
\item [(2)] $\dot{H}^{s_c}$-norm is invariant for Eq. (\ref{eq1}), i.e., $\|u^{\lambda}\|_{\dot{H}^{s_c}}= \|u\|_{\dot{H}^{s_c}}$,
where $s_c=\frac N2-\frac{2s}{p-1}$.
\end{itemize}

Recently, the Cauchy problem  (\ref{eq1})-(\ref{1.2})
has been widely studied in the recent years but is not completely settled yet, see, e.g.
\cite{cho-hwang-kwon-leedcds15} and \cite{hong-sirecpaa15}.
Let $N\geq 2$, $\frac12\leq s<1$ and $1+\frac{4s}{N}<p<1+\frac{4s}{N-2s}$. If $u_{0}\in H^s$, then
  the Cauchy problem (\ref{eq1})-(\ref{1.2}) has  a unique solution $u(t,x)$
on $I=[0,T)$ satisfying $u(t,x)\in C(I;H^{s})\bigcap
C^1(I;H^{-s})$. Moreover,  either $T=+\infty$ (global existence) or
both $0<T<+\infty$ and $\lim\limits_{t\to T} \|
u(t,x)\|_{H^{s}} =+\infty$ (blow-up). Furthermore,   $\forall\ t\in I$, $u(t,x)$ has two important  conservation
laws.
\begin{itemize}
\item [(i)] Conservation of energy:
\begin{equation}\label{E}E[u(t)]= \frac 12 \int _{\mathbb{R}^N}\overline{u}(-\triangle )^{s}udx -\frac{1}{p+1}\int_{\mathbb{R}^N}|u(x,t)|^{p+1}dx=E[u_0].\end{equation}
 \item [(ii)] Conservation of mass:
 \begin{equation}\label{M}M[u(t)]=\int_{\mathbb{R}^N} |u(t,x)|^2dx=M[u_0].\end{equation}
\end{itemize}
Guided by a analogy to classical NLS, the sufficient criteria for blowup of the solution can be found  in \cite{b-h-ljfa16} in terms of quantities
of the ground states $Q\in H^s(\mathbb R^N)$, by solving
 \be\label{eqQ}(-\triangle
)^s Q+Q-|Q|^{p-1}Q=0,\ \  \ \  Q\in H^s(\mathbb{R}^N)\ee
and   the Gagliardo-Nirenberg inequality (see Theorem 3.2 in \cite{Zhu20162})
\begin{equation}
\label{G-N}\int|v(x)|^{p+1}dx\leq C_{GN}
\left\|
v\right\|_2^{p+1-\frac{N(p-1)}{2s}}\left\|v\right\|_{\dot{H}^s}^{\frac{N(p-1)}{2s}}\end{equation}
with
\begin{equation}\label{CGN}
C_{GN}=\frac{2s(p+1)}{N(p-1)}\frac1{\left\|
Q\right\|_2^{p+1-\frac{N(p-1)}{2s}}\left\|Q\right\|_{\dot{H}^s}^{\frac{N(p-1)}{2s}-2}}.
\end{equation}

The blow-up and long-time dynamics of
the fractional NLS turn out to be very interesting problems.
To the best of the authors' knowledge, the cases that have been successfully
addressed by now are:
i) for the fractional NLS with nonlocal Hartree-type nonlinearites and radial data, see, e.g. \cite{ChoHwangHajaiejOzawa2012,Lenzmann2007}. Recently,  Guo and Zhu \cite{g-zarxiv1}
obtained a  sharp threshold of the scattering versus blow-up for the focusing $L^2$-supercritical case.
ii) for the power-type nonlinearities,  Boulenger, Himmelsbach, Lenzmann \cite{b-h-ljfa16} derived
a general blowup result for \eqref{eq1} in both $L^2$-supercritical and $L^2$-critical cases respectively,
subject to certain threshold. Recently, the authors in \cite{guo-sire-wang-zhaoarxiv1310.6816} performed Kenig-Merle type
argument \cite{km} to show the global well-posedness of radial solutions and scattering below sharp threshold of ground state solutions.
In \cite{sun-zhengarxiv}, the authors adapt the strategy in \cite{dodson-murphy} to prove a similar scattering result for the 3D radial focusing cubic
fractional NLS, under the restriction that $s\in(\frac34,1)$.

In this paper, we give a complement of the blowup result given by Boulenger, Himmelsbach, Lenzmann \cite{b-h-ljfa16} for general
dimensions and nonlinearities for $s\in(\frac N{2N-1},1)$, with different method from the 3D cubic case.
More precisely, we obtain the  scattering for the $L^2$-supercritical NLS Eq. (\ref{eq1}) in terms
of the arguments in \cite{qgcpde16,radial,km}, as follows.


\begin{theorem}\label{th1}
Let $N\geq 2$ and $1+\frac{4s}{N}<p<1+\frac{4s}{N-2s}$ Suppose that  $u_0\in H^s$ is radial and  $ M[u_{0}]^{\frac{s-s_c}{s_c}}E[u_{0}]<M[Q]^{\frac{s-s_c}{s_c}}E[Q],
$ where $Q$ is the ground-state solution of \eqref{eqQ}.
 If $\frac N{2N-1}\leq s<1$ and
  \begin{equation*}M[u_{0}]^{\frac{s-s_c}{s_c}}\|
u_{0}\|^2_{\dot H^s}<M[Q]^{\frac{s-s_c}{s_c}}\| Q\|^2_{\dot H^s},\end{equation*}
  then the corresponding solution $u(t)$ of
(\ref{eq1})-(\ref{1.2}) exists globally in $H^s$. Moreover, $u(t)$ scatters in $H^s$. Specifically, there
exists $\phi_\pm\in H^s$ such that
$\lim\limits_{t\rightarrow\pm\infty}\|u(t)-e^{-it(-\Delta)^{s}}\phi_\pm\|_{H^s}=0$.
\end{theorem}

We should point out that the sharp criteria of scattering  for the nonlinear Schr\"{o}dinger equation is a quite important and interesting problems, and many researchers have devoted on this topics (see e.g. \cite{Cazenave2003,dodson-murphy,qgcpde16,g-zarxiv1,radial,km,sun-zhengarxiv}). The scattering  involves in the Strichartz estimates and the choice of admissible pairs, which is quite different and difficult with respect to different nonlinearities.   Although in \cite{g-zarxiv1}, we have proved the scattering for the fractional Hartree equation in the $L^2$ supercritical case, that   for the  fractional NLS (\ref{eq1}) with  power-type  nonlinearity is a nontrivial extension(e.g. Proposition \ref{h1scattering}, Theorem \ref{rigidity}).

At the end of this section, we  introduce  some notations.
$L^q:=L^q(\mathbb{R}^N)$,
$\|\cdot\|_q:=\|\cdot\|_{L^q(\mathbb{R}^N)}$,
the time-space mixed norm $$\|u\|_{L^qX}:=\left(\int_\mathbb R\|u(t,\cdot)\|_{X}^qdt\right)^{\frac1q},$$
$H^s:=H^s(\mathbb{R}^N)$, $\dot{H}^s:=\dot{H}^s(\mathbb{R}^N)$, and
$\int \cdot dx:=\int_{\mathbb{R}^N}\cdot dx$.  $\mathcal{F}v=\widehat{v}$ denotes  the Fourier
transform of $v$, which for $v\in L^1( \mathbb{R}^N)$ is given by
$\mathcal{F}v= \widehat{v} (\xi):= \int e^{-i
x\cdot\xi}v(x)dx$  for all $\xi\in \mathbb{R}^N$, and $\mathcal{F}^{-1}v$ is
 the inverse Fourier transform  of
$v(\xi)$. $\Re z$ and $\Im z$ are the
real and imaginary parts of the complex number $z$, respectively.
 $\overline{z}$ denotes the complex conjugate of the complex number $z$. The various positive
constants will be denoted by $C$ or $c$.

\newpage

\section{Local theory and Strichartz estimate }

In fact, the Cauchy
problem \eqref{eq1}-\eqref{1.2} has the following integral equation:
$$u(t) = U(t)u_0+i\int_0^tU(t-t^1)|u|^{p-1}u(t^1)dt^1
$$
where$$
U(t)\phi(x) = e^{-i(-\triangle)^st}\phi(x)=\frac{1}{(2\pi)^{\frac N2}}\int  e^{i(x\cdot\xi-|\xi|^{2s})}\widehat{\phi}(\xi)d\xi.$$
First, we  recall the local theory for Eq.~\eqref{eq1}~ by the radial Strichartz estimate
~(\cite{G-W}).

\begin{definition}\label{defadmissible}
For the given $\theta\in[0,s)$, we state that the pair $(q,r)$ is $\theta$-level admissible,
denoted by $(q,r)\in\Lambda_\theta$,  if
\begin{align}\label{gap}
q,r\geq2,\ \
\frac{2s}{q}+\frac N{r}=\frac N2-\theta
\end{align}
and
 \begin{align}\label{range}
 \frac{4N+2}{2N-1}\leq q\leq\infty,\ \
\frac1{q}\leq\frac {2N-1}2(\frac12-\frac1{r}),\ \ \
or\ \  2\leq q<\frac{4N+2}{2N-1},
  \ \frac1{q}<\frac {2N-1}2(\frac12-\frac1{r}).
 \end{align}
Correspondingly, we denote the dual $\theta$-level admissible pair by $(q',r')\in\Lambda'_{\theta}$~if
~$(q,r)\in\Lambda_{-\theta}$~with~$(q',r')$~is the H\"older~ dual to~$(q,r).$~
\end{definition}
\begin{proposition}\label{prostrichartz} (see\cite{G-W})
Assume that $N\geq2$ and that $u_0,f$ are radial; then  for $q_j,r_j\geq2,j=1,2$,
\be
\|U(t)\phi\|_{L^{q_1}L^{r_1}}\leq C\|D^\theta\phi\|_2,
\ee
where $D^\theta=(-\triangle)^{\frac\theta2}$,
\be\label{inhomo}
\|\int_0^tU(t-t^1)f(t^1)dt^1\|_{L^{q_1}L^{r_1}}\leq C\|f\|_{L^{q'_2}L^{r'_2}},
\ee
in which $\theta\in\mathbb R$, the pairs $(q_j,r_j)$ satisfy the range conditions \eqref{range}
and the gap condition
$$\frac{2s}{q_1}+\frac N{r_1}=\frac N2-\theta,\ \ \ \frac{2s}{q_2}+\frac N{r_2}=\frac N2+\theta.$$
\end{proposition}
\begin{definition}\label{defnorm}

We  define the following Srichartz norm
$$\|u\|_{S(\Lambda_{s_c})}=\sup_{(q,r)\in\Lambda_{s_c}}\|u\|_{L^qL^r}
 $$
and the dual Strichartz norm
$$\|u\|_{S'(\Lambda_{-s_c})}=\inf_{(q',r')\in\Lambda'_{s_c}}\|u\|_{L^{q'}L^{r'}}=\inf_{(q,r)\in\Lambda_{-s_c}}\|u\|_{L^{q'}L^{r'}},$$
where~$(q',r')$~is the H\"older~ dual to~$(q,r).$~
\end{definition}

\begin{remark}
Notice that if $$s\in[\frac N{2N-1},1)\subset(\frac12,1),$$
the gap condition \eqref{gap} with $\theta=0$
right implies the range condition \eqref{range}, which further means that $\Lambda_0$ is nonempty.
That is we have a full set of 0-level admissible Strichartz estimates without loss of derivatives in radial case.
By taking
\begin{align}\label{qc}q_c=r_c=\frac{(p-1)(N+2s)}{2s},\end{align}
 we see that $(q_c,r_c)\in \Lambda_{s_c}\neq\emptyset$
 is an $s_c$-level admissible pair.

\end{remark}

When $\phi,f$ are radial, from Proposition \ref{prostrichartz},  we have the following Strichartz estimates.
$$\|U(t)\phi\|_{S(\Lambda_0)}\leq C\|\phi\|_2$$ and
$$\left\|\int_0^tU(t-t^1)f(\cdot,t^1)dt^1\right\|_{S(\Lambda_0)}\leq C\|f\|_{S'(\Lambda_0)}.$$
Then,  we   further obtain
$$\|U(t)\phi\|_{S(\Lambda_{s_c})}\leq c\|\phi\|_{\dot{H}^{s_c}}, \ \ \left\|\int_0^tU(t-t^1)f(\cdot,t^1)dt^1\right\|_{S(\Lambda_{s_c})}\leq C\|D^{s_c}f\|_{S'(\Lambda_{0})}$$
and
$$\left\|\int_0^tU(t-t^1)f(\cdot,t^1)dt^1\right\|_{S(\Lambda_{s_c})}\leq C\|f\|_{S'(\Lambda_{-s_c})},$$ where we use
the Sobolev embedding.

Next, we denote~$S(\Lambda_{\theta};I)$~to indicate  its restriction to a time subinterval~$I\subset(-\infty,+\infty).$

\begin{proposition}\label{sd}
(Small data) Let $N\geq 2$ and $1+\frac{4s}{N}<p<1+\frac{4s}{N-2s}$.
If $\|u_{0}\|_{\dot H^{s_c}}\leq A$ is radial, then,
there exists  $\delta_{sd}=\delta_{sd}(A)>0$ such that when
~$\|U(t)u_{0}\|_{S(\Lambda_{s_c})}\leq \delta_{sd}$,  the corresponding solution $u=u(t) $~solving \eqref{eq1} is global, and
\[
 \|u\|_{S(\Lambda_{s_c})}\leq
2\|U(t)u_{0}\|_{S(\Lambda_{s_c})},\qquad \|D^{s_c}u\|_{S(\Lambda_0)}\leq
2c\|u_{0}\|_{\dot H^{s_c}}.\]
(Note that by
 the Strichartz estimates, the hypotheses are satisfied if
~$\|u_{0}\|_{\dot H^{s_c}}\leq c\delta_{sd}. $)
\end{proposition}

\begin{proof}
Denote
$$\Phi_{u_0}(v)=U(t)u_{0}+i\int_0^tU(t-t^1)|v|^{p-1}v(t^1)dt^1.$$
It follows from  the Strichartz estimates  that
$$\|D^{s_c}\Phi_{u_0}(v)\|_{S(\Lambda_0)}\leq c\|u_{0}\|_{\dot{H}^{s_c}}
+c\|D^{s_c}[|v|^{p-1}v]\|_{L^{q^{'}}L^{r^{'}}}$$
and
$$\|\Phi_{u_0}(v)\|_{S(\Lambda_{s_c})}\leq \|U(t)u_{0}\|_{S(\Lambda_{s_c})}
+\||v|^{p-1}v\|_{L^{\frac {q_2}p}L^{\frac {r_2}p}}\leq
 \|U(t)u_{0}\|_{S(\Lambda_{s_c})}
+\|v\|^{p}_{L^{q_2}L^{r_2}},$$ where ~$(q',r')\in\Lambda'_{0},$ $(q_2,r_2)\in\Lambda_{s_c}$ and  $(\frac {q_2}p,\frac {r_2}p)\in\Lambda'_{s_c}$.
Then, by applying the fractional Leibnitz~\cite{ChoHwangHajaiejOzawa2012,kato1995}~,   we deduce that
\begin{align*}
\|D^{s_c}[|v|^{p-1}v]\|_{L^{q^{'}}L^{r^{'}}}
& \leq c\||u|^{p-1}\|_{L^{\frac{q_1q'}{q_1-q'}}L^{\frac{r_1r'}{r_1-r'}}}\|D^{s_c}v\|_{L^{q_{1}}L^{r_{1}}}\\
&\leq c\|v\|^{p-1}_{L^{q_{2}}L^{r_{2}}}\|D^{s_c}v\|_{L^{q_{1}}L^{r_{1}}},
\end{align*}
where the pairs $(q,r),(q_1,r_1)\in\Lambda_{0},$
$(q_2,r_2)\in\Lambda_{s_c}$.
Now, we take  $$\delta_{sd}\leq\left(\min\left(\frac{1}{2^pc},\frac{1}{2^p}\right)\right)^{\frac1{p-1}},$$
and define $$B:=\left\{v|\|v\|_{S(\Lambda_{s_c})}\leq 2\|U(t)u_{0}\|_{S(\Lambda_{s_c})},
\|D^{s_c}v\|_{S(\Lambda_0)}\leq
2c\|u_{0}\|_{\dot{H}^{s_c}}
\right\}.$$
Then, we can prove that ~$\Phi_{u_0}$ is a contraction mapping from $B$ to $B$, which completes the proof.

\end{proof}

\begin{proposition}\label{h1scattering}
(Scattering criterion) Let $N\geq 2$ and $1+\frac{4s}{N}<p<1+\frac{4s}{N-2s}$.
If $u_{0}\in H^s$ is radial and $u(t)$ is global with  both bounded $s_c$-level Strichartz norm $\|u\|_{S(\Lambda_{s_c})}<\infty$ and
uniformly bounded $H^s$ norm
$\sup\limits_{t\in[0,\infty)}\|u\|_{H^s}\leq B,$
then $u(t)$ scatters in $H^s$ as $t\rightarrow +\infty.$
More precisely, there exists $\phi^{+}\in H^s$ such that
$$\lim_{t\rightarrow+\infty}\|u(t)-U(t)\phi^{+}\|_{H^s}=0.$$
\end{proposition}

\begin{proof}
It follows from the integral equation
\begin{align}\label{203}
u(t)=U(t)u_{0}+i\int_0^tU(t-t^1)|u|^{p-1}u(t^1)dt^1
\end{align}
that
\begin{align}\label{2.03}
u(t)-U(t)\phi^{+}=-i\int_t^{\infty}U(t-t^1)|u|^{p-1}u(t^1)dt^1,
\end{align}
where $$\phi^{+}=u_{0}+i\int_0^{\infty}U(-t^1)|u|^{p-1}u(t^1)dt^1.$$

Applying Proposition \ref{prostrichartz}, we deduce that
 for $0\leq\alpha\leq s$,
there exist some $(q,r)\in\Lambda_0$, $(q_1,r_1)\in\Lambda'_{0}$ such that
\begin{align}\label{204}
 \left\|D^\alpha\left(\int_{I}U(t-s)\left(|u|^{p-1}u(s,x)\right)ds\right)\right\|_{L^{q}_{I}L^{r}}
&\leq C\|D^\alpha\left(|u|^{p-1}u\right)\|_{L^{q_1}_{I}L^{r_1}}\\ \nonumber
&\leq C\|D^\alpha u\|_{L^q_{I} L^r}\|u\|_{L^{q_c}_{I}L^{r_c}}^{p-1},
\end{align}
where $I\subset[0,+\infty)$, $$\frac1{q_1}=\frac{p-1}{q_c}+\frac1q,\ \ \frac1{r_1}=\frac1r+
\frac{p-1}{r_c}.$$
Due to  $\|u\|_{L^{q_c}_{[0,\infty)}L^{r_c}}<\infty$, we  divide  $[0,+\infty)$ into $N$ subintervals: $I_j=[t_j,t_{j+1}],1\leq j\leq N$, such that
  $\|u\|_{L^{q_c}_{I_j}L^{r_c}}<\delta$ (for small $\delta$ ) on each subinterval $I_j$. Thus, from \eqref{203} and \eqref{204}, we see that
for $0\leq\alpha\leq s, \forall 1\leq j\leq N$,
\begin{align*}
\|D^\alpha u\|_{L^{q}_{I_j}L^{r}}
&\leq
\|U(t)u(t_j)\|_{L^{q}_{I_j}L^{r}}
+\left
\|D^\alpha\left(\int_{I_j}U(t-s)\left(|u|^{p-1}u(s,x)\right)ds\right)\right\|_{L^{q}_{I_j}L^{r}}\\
&\leq \|U(t)u(t_j)\|_{L^{q}_{I_j}L^{r}}+ C\|D^\alpha u\|_{L^q_{I_j} L^r}\|u\|_{L^{q_c}_{I_j}L^{r_c}}^{p-1}\\
&\leq
CB+C\delta^{p-1}\|D^\alpha u\|_{L^q_{I_j} L^r}.
\end{align*}
Let $\delta$ be small and  satisfy $C\delta^{p-1}<\frac12$. Then
$\|D^\alpha u\|_{L^{q}_{I_j}L^{r}}<\infty,\ 1\leq j\leq N$, and $$\|D^\alpha u\|_{L^{q}L^{r}}<\infty.$$Moreover, from  \eqref{2.03}, we see that  for  $0\leq\alpha\leq s$,
\begin{equation}\label{0001}
\|D^\alpha (u(t)-U(t)\phi^{+})\|_{2}\leq \|u\|_{L^{q_c}_{[t,\infty)}L^{r_c}}^{p-1}\|D^\alpha u\|_{L^{q}_{[t,\infty)}L^{r}}.
\end{equation}
Therefore, we can obtain  the claim by taking $\alpha=0$ and $\alpha=s$ in (\ref{0001})  and
letting  $t\rightarrow+\infty$.
\end{proof}

\begin{proposition}\label{properturb}
 For any given $A$, there exist $\epsilon_0=\epsilon_0(A,N,p)$ and
$c=c(A)$ such that for any $\epsilon\leq\epsilon_0$, any interval $I
=(T_1,T_2)\subset \mathbb R$ and any
 $\tilde{u}=\tilde{u}(x,t)\in H^s$ satisfying
\begin{equation}\label{0002} i\tilde{u}_{t}-(-\Delta)^s \tilde{u}-|\tilde{u}|^{p-1}\tilde{u}=e.\end{equation}
If $$\|\tilde{u}\|_{S(\Lambda_{s_c})}\leq A,\ \   \|e\|_{S'(\Lambda_{-s_c})}\leq \epsilon\ \
{\rm and}\ \ \|e^{-i(t-t_0)(-\Delta)^s}(u(t_0)-\tilde{u}(t_0)\|_{S(\Lambda_{s_c})}\leq \epsilon,$$
then$$\|u\|_{S(\Lambda_{s_c})}\leq c=c(A)<\infty.$$
\end{proposition}

\begin{proof}
Let   $u=\tilde u+w$, where $\tilde u$ is the solution of (\ref{0002}) and   $w$ is the solution of
\begin{align}\label{w}
i\partial_t w-(-\Delta)^sw-|w+\tilde u|^{p-1}(w+\tilde u)+|\tilde u|^{p-1}\tilde u+e=0.
\end{align}
For any $t_0\in I$,  $I=(T_1,t_0]\cup[t_0,T_2)$.  We need only
consider on $I_+=[t_0,T_2)$, since the case on $I_-=(T_1,t_0]$ can be
considered similarly.  Since $\|\tilde{u}\|_{S(\Lambda_{s_c})}\leq A$, we can
partition $[t_0,T_2)$ into $N=N(A)$ intervals $I_j =[t_j,t_{j+1}]$
such that for each $j$, the quantity
$\|\tilde{u}\|_{S(\Lambda_{s_c};I_j)}<\delta$ is suitably small with $\delta$
to be chosen later. The integral equation of $w$ with initial time
$t_j$ is
\begin{align}\label{w2}
  w(t)=e^{-i(t-t_j)(-\Delta)^s}w(t_j)-i\int_{t_j}^te^{-i(t-s)(-\Delta)^s}[|w+\tilde u|^{p-1}(w+\tilde u)-|\tilde u|^{p-1}\tilde u-e](s)ds.
\end{align}
Using the inhomogeneous Strichartz estimates \eqref{inhomo} on $I_j$, we obtain that for some $(q_1,r_1)\in\Lambda_{-s_c}$,
\begin{align*}
&\|w\|_{S(\Lambda_{s_c};I_j)}\\
\leq&\|e^{-i(t-t_j)(-\Delta)^s}w(t_j)\|_{S(\Lambda_{s_c};I_j)}
+c\||w+\tilde u|^{p-1}(w+\tilde u)+|\tilde u|^{p-1}\tilde u\|_{L^{q_1'}(I;L^{r_1'})}+\|e\|_{S'(\Lambda_{-s_c})}\\
\leq&\|e^{-i(t-t_j)(\Delta)^s}w(t_j)\|_{S(\Lambda_{s_c};I_j)}
+c\| \tilde u\|^{p-1}_{S(\Lambda_{s_c};I_j)}\|w\|_{S(\Lambda_{s_c};I_j)}+c\|w\|_{S(\Lambda_{s_c};I_j)}^p+\|e\|_{S'(\Lambda_{-s_c})}\\
\leq&\|e^{-i(t-t_j)(-\Delta)^s}w(t_j)\|_{S(\Lambda_{s_c};I_j)} +c\delta^{p-1}\|w\|_{S(\Lambda_{s_c};I_j)}+c\|w\|_{S(\Lambda_{s_c};I_j)}^p+c\epsilon_0.
\end{align*}
If
\begin{align}\label{condition}
\delta\leq\left(\frac1{4c}\right)^{\frac1{p-1}},\ \ \ \
\|e^{-i(t-t_j)(-\Delta)^s}w(t_j)\|_{S(\Lambda_{s_c};I_j)}+c\epsilon_0\leq\frac12\left(\frac1{4c}\right)^{\frac1{p-1}},
\end{align}
then 
\begin{align*}
\|w\|_{S(\Lambda_{s_c};I_j)}\leq2\|e^{-i(t-t_j)(-\Delta)^s}w(t_j)\|_{S(\Lambda_{s_c};I_j)}+c\epsilon_0.
\end{align*}
Now take $t=t_{j+1}$ in \eqref{w2}, and apply $e^{-i(t-t_{j+1})(-\Delta)^s}$
to both sides to obtain
\begin{align*}
 e^{-i(t-t_{j+1})(-\Delta)^s} w(t_{j+1})&=e^{-i(t-t_{j})(-\Delta)^s}w(t_j)\\
 &-i\int_{t_j}^{t_{j+1}}e^{-i(t-s)(-\Delta)^s}[|w+\tilde u|^{p-1}(w+\tilde u)-|\tilde u|^{p-1}\tilde u-e](s)ds.
\end{align*}
Since the Duhamel integral is confined to $I_j$, using the inhomogeneous Strichart'z estimates
and following a similar argument as above, we obtain that
\begin{align*}
&\| e^{-i(t-t_{j+1})(-\Delta)^s}
w(t_{j+1})\|_{S(\Lambda_{s_c};I_+)}\\&\leq\|e^{-i(t-t_{j})(-\Delta)^s}w(t_j)\|_{S(\Lambda_{s_c};I_+)}
+c\delta^{p-1}\|w\|_{S(\Lambda_{s_c};I_j)}+c\|w\|_{S(\Lambda_{s_c};I_j)}^p+c\epsilon_0\\
&\leq2\|e^{-i(t-t_j)(-\Delta)^s}w(t_j)\|_{S(\Lambda_{s_c};I_j)}+c\epsilon_0.
\end{align*}
Iterating beginning with $j=0$, we obtain
\begin{align*}
\| e^{-i(t-t_{j})(-\Delta)^s}
w(t_{j})\|_{S(\Lambda_{s_c};I_+)}\leq2^{j}\|e^{-i(t-t_{0})(-\Delta)^s}w(t_0)\|_{S(\Lambda_{s_c};I_j)}
+(2^{j}-1)c\epsilon_0 \leq2^{j+2}c\epsilon_0.
\end{align*}
To accommodate the conditions \eqref{condition} for all intervals $I_j$ with $0\leq j\leq N-1$, we require
\begin{align}\label{condition'}
2^{N+2}c\epsilon_0\leq(\frac1{4c})^{\frac1{p-1}}.
\end{align}
Finally,$$\|w\|_{S(\Lambda_{s_c};I_+)}\leq
\sum_{j=0}^{N-1}2^{j+2}c\epsilon_0+cN\epsilon_0\leq
c(N)\epsilon_0,$$ which implies $\|w\|_{S(\Lambda_{s_c};I_+)}\leq c(A)\epsilon_0$
since $N=N(A)$, concluding the proof.

\end{proof}

\vspace{0.5cm}

\section{Variational Characteristic and Invariant Sets}

First, we collect some variational properties of $Q$, as follows.
\begin{lemma}  (\cite{Zhu20162})  \label{reQ} Let $N\geq 2$, $0<s<1$ and  $1+\frac{4s}{N}<p<1+\frac{4s}{N-2s}$.
 Suppose that
 $Q$ is  the ground-state
solution  of   (\ref{eqQ}). Then, we have  the following properties.
\[\begin{array}{lll}\vspace{0.3cm}

 & (i) \quad \|Q\|_{p+1}^{p+1}=\frac{2s(p+1)}{N(p-1)}\|Q\|_{\dot{H}^s}^2=\frac{2s(p+1)}{2s(p+1)-N(p-1)}\|Q\|_2^2.\\
 \vspace{0.3cm}

& (ii)\quad E[Q]=\frac 12 \int \overline{Q}(-\triangle)^s  Q dx-\frac {1}{p+1} \|Q\|_{p+1}^{p+1}=\frac{N(p-1)-4s}{2N(p-1)}\|Q\|_{\dot{H}^s}^2.\\
\vspace{0.3cm}

&(iii) \quad E[Q]M[Q]^{\frac{s-s_c}{s_c}}=\frac{N(p-1)-4s}{4s-(N-2s)(p-1)}\|Q\|_2^{\frac{2s}{s_c}}.\\
\vspace{0.3cm}

&(iv)\quad \|Q\|_{\dot{H}^s}^{2}M[Q]^{\frac{s-s_c}{s_c}}=\frac{N(p-1)}{4s-(N-2s)(p-1)}\|Q\|_2^{\frac{2s}{s_c}}.\\
\vspace{0.3cm}

&(v)\quad  C_{GN}=\frac{\|Q\|_{p+1}^{p+1}}{\|Q\|_2^{p+1-\frac{N(p-1)}{2s}}\|Q\|_{\dot H^s}^{\frac{N(p-1)}{2s}}}
=\frac{2s(p+1)}{N(p-1)}\frac{1}{\|Q\|_2^{p+1-\frac{N(p-1)}{2s}}\|Q\|_{\dot H^s}^{\frac{N(p-1)}{2s}-2}}.
\end{array}  \]
\end{lemma}

\begin{remark}
 In fact, Caffarelli and Silvestre in \cite{Caffarelli-Silvestre2007} first proposed  a  general fractional Laplacian.  And  then many researchers have been studying  the time dependent and independent of  fractional  nonlinear  Schr\"{o}dinger  equations   (see\cite{Chen-Li-Li2017,Felmer-Quaas-Tan2012,g-h17na,g-h17jde,Liu-Ma2016,ZhangZhu2015,Zhu2016}).

\end{remark}

  Let $u\in H^{s}\setminus\{0\}$, and define
\[K_g=\{\|  u \|_{\dot{H}^s}^{2}M[u]^{\frac{s-s_c}{s_c}}< \|Q\|_{\dot{H}^{s}}^{2}M[Q]^{\frac{s-s_c}{s_c}},\  E[u]M[u]^{\frac{s-s_c}{s_c}}<  E[Q]M[Q]^{\frac{s-s_c}{s_c}}\}.\]
\begin{proposition} \label{invariant set}  Let $N\geq 2$, $0<s<1$ and  $1+\frac{4s}{N}<p<1+\frac{4s}{N-2s}$.
Let $Q$ be the ground-state solution of
(\ref{eqQ}).  Then
$K_g$  is invariant manifold of (\ref{eq1}).
\end{proposition}

\begin{proof}
It follows from the conservation of energy and the sharp Gagaliardo-Nirenberg inequality  \eqref{G-N} that \begin{align*}
M[u]^{\frac{s-s_c}{s_c}}E[u]=&\frac12 \|u(t)\|^{\frac{2(s-s_c)}{s_c}}_{2}\|D^s u(t)\|^2_{2}
-\frac1{p+1}\|u\|_{p+1}^{p+1}\|u\|_2^{\frac{2(s-s_c)}{s_c}}\\
\geq& \frac12 (\|u(t)\|^{\frac{s-s_c}{s_c}}_{2}\|D^s u(t)\|_{2})^2
-\frac{C_{GN}}{p+1}(\|u(t)\|^{\frac{s-s_c}{s_c}}_{2}\|D^s u(t)\|_{2})^{\frac{\gamma}{s}}.
\end{align*}
Now, we define $f(y)=\frac12y^2-\frac1{p+1}C_{GN}y^{\frac{N(p-1)}{2s}}$.  We find that  $f(y)$ has the following properties:
$f'(y)=y\left(1-C_{GN}\frac{N(p-1)}{2s(p+1)}y^{\frac{N(p-1)-4s}{2s}}\right)$, and
thus,  $y_0=0$ and
$y_1=\|Q\|^{\frac{s-s_c}{s_c}}_{2}\|D^s Q\|_{2}$ are two roots of $f'(y)=0$, which implies that  $f$ has a local minimum at $y_0$ and a local maximum at $y_1$. From Lemma \ref{reQ}, we have
$f_{max}=f(y_1)=M[Q]^{\frac{s-s_c}{s_c}}E[Q]$, and for all $t\in I$ \begin{align}\label{1}
f(\|u(t)\|^{\frac{s-s_c}{s_c}}_{2}\|D^s u(t)\|_{2})\leq M[u(t)]^{\frac{s-s_c}{s_c}}E[u(t)]=M[u_0]^{\frac{s-s_c}{s_c}}E[u_0]
<f(y_1).
\end{align}

If   $u_0\in K_g$, i.e., $\|u_{0}\|^{\frac{s-s_c}{s_c}}_{2}\|D^s
u_{0}\|_{2}<y_1$, then by  \eqref{1} and the continuity of $\|D^s
u(t)\|_2$ in $t$, we claim that
 $t\in \mathbb{R},$~
\begin{equation}\label{2.3'}
\|  u(t) \|_{\dot{H}^s}^{2}M[u(t)]^{\frac{s-s_c}{s_c}}< \|Q\|_{\dot{H}^{s}}^{2}M[Q]^{\frac{s-s_c}{s_c}}.
\end{equation}
Indeed, if (\ref{2.3'}) is not true,
there must be  $t_1\in I$ such that $\|u(t_1)\|^{\frac{s-s_c}{s_c}}_{2}\|D^s
u(t_1)\|_{2}\geq y_1$.
Since $u(t,x)\in C(I;H^s)$ is
continuous with respect to $t$, we can find a  $0<t_0\leq t_1$ such
that $\|u(t_0)\|^{\frac{s-s_c}{s_c}}_{2}\|D^s
u(t_0)\|_{2}= y_1$. Thus, by injecting
$E[u(t_0)]=E[u_0]$ and $\|u(t_0)\|^{\frac{s-s_c}{s_c}}_{2}\|D^s
u(t_0)\|_{2}= y_1$ into  (\ref{1}), we see that
\[f(y_1)>M[u_0]^{\frac{s-s_c}{s_c}}E[u_0]=M[u(t_0)]^{\frac{s-s_c}{s_c}}E[u(t_0)]\geq f(\|u(t_0)\|^{\frac{s-s_c}{s_c}}_{2}\|D^s
u(t_0)\|_{2})=f(y_1).\]
This is a contradiction.  This completes the proof.

\end{proof}

\begin{remark}\label{remdelta} In fact, using the same argument in Proposition \ref{invariant set}, we can obtain a precise estimate.
Specially,  if the initial data is such that $$\|  u_0 \|_{\dot{H}^s}^{2}M[u_0]^{\frac{s-s_c}{s_c}}< \|Q\|_{\dot{H}^{s}}^{2}M[Q]^{\frac{s-s_c}{s_c}}, $$ then, we can chose a $\delta>0$ such that
$M[u]^{\frac{s-s_c}{s_c}}E[u]<(1-\delta)M[Q]^{\frac{s-s_c}{s_c}}E[Q]$.
Moreover, for the solution $u=u(t)$ with the above initial data we can find  $\delta_0=\delta_0(\delta)$ such that
$\|u(t)\|^{\frac{s-s_c}{s_c}}_{2}\|D^su(t)\|_{2}<
(1-\delta_0)\|Q\|^{\frac{s-s_c}{s_c}}_{2}\|D^s Q\|_{2}$.
\end{remark}

By the invariance of $K_g$, we see that (\ref{2.3'}) is true.
 In
particular, the $H^s$-norm of the solution $u$ is bounded, which
proves the global existence of the solution in this case.

\begin{theorem}\label{th3.5} Let $N\geq 2$, $0<s<1$ and  $1+\frac{4s}{N}<p<1+\frac{4s}{N-2s}$. Assume that $u_{0}\in H^{s}$, and   $I=(T_{-},T_{+})$  is the
maximal existence interval of  ~$u(t)$~ solving ~\eqref{eq1}.
 If  $u_0\in K_g$,
then ~$I=(-\infty,+\infty)$,~ i.e., ~$u(t)$~ exists globally in
time.
\end{theorem}

\begin{lemma}\label{lemlowerbound}
Let $u_0\in K_g$. Furthermore,
take $\delta>0$ such that
$M[u_0]^{\frac{s-s_c}{s_c}}E[u_0]<(1-\delta)M[Q]^{\frac{s-s_c}{s_c}}E[Q]$.
If $u$ is a solution to problem \eqref{eq1} with initial data $u_0$,
then there exists $C_\delta>0$ such that for all $t\in\mathbb R$,
\begin{align}\label{4.7}
\|D^s u\|_2^2-\frac{N(p-1)}{2s(p+1)}\|u\|_{p+1}^{p+1}\geq
C_\delta\|D^s u\|_2^2.
\end{align}
\end{lemma}
\begin{proof}
Let $\delta_0=\delta_0(\delta)>0$ be defined in Remark \ref{remdelta},.
Then, for all $t\in\mathbb R$, we have
\begin{align}\label{low}
\|u(t)\|^{\frac{s-s_c}{s_c}}_{2}\|D^s u(t)\|_{2}<
(1-\delta_0)\|Q\|^{\frac{s-s_c}{s_c}}_{2}\|D^s Q\|_{2}.
\end{align} Denote
$$H(t)=\frac1{\|Q\|^{\frac{2(s-s_c)}{s_c}}_{2}\|D^s Q\|^2_{2}}(
\|u(t)\|^{\frac{2(s-s_c)}{s_c}}_{2}\|D^s
u(t)\|^2_{2}-\frac{N(p-1)}{2s(p+1)}\|u\|_{p+1}^{p+1}\|u(t)\|^{\frac{2(s-s_c)}{s_c}}_{2})$$
and $G(y)=y^2-y^{\frac{N(p-1)}{2s}}$. Applying  the sharp Gagliardo-Nirenberg
inequality in \eqref{G-N}, we deduce that
$$H(t)\geq G \left(\frac{\|u(t)\|^{\frac{s-s_c}{s_c}}_{2}\|D^s u(t)\|_{2}}{\|Q\|^{\frac{s-s_c}{s_c}}_{2}\|D^s Q\|_{2}}\right).$$
When $0\leq y\leq1-\delta_0$, by the properties of the function $G(y)$, we deduce that there exists a constant $C_\delta$ such that $g(y)\geq
C_\delta y^2$ provided $0\leq y\leq1-\delta_0$.  This completes the proof.

\end{proof}

\begin{lemma}\label{lemcompare}
(Comparability of  gradient and energy)
Let $u_0\in K_g$. Then,
$$\frac{N(p-1)-4s}{2N(p-1)}\|D^s u(t)\|_{2}^2\leq E[u(t)]\leq\frac{1}{2}\|D^su(t)\|_{2}^2.$$
\end{lemma}

\begin{proof} The second inequality can be obtained directly by
the expression of $E[u(t)]$. The first inequality  follows from \eqref{G-N}, \eqref{CGN} and \eqref{2.3'} that
\begin{align*}\frac{1}{2}\|D^s u\|_{L^2}^2-\frac{1}{p+1}\|u\|_{p+1}^{p+1}&\geq\frac{1}{2}\|D^s  u\|_{L^2}^2
\left(1-\frac{4s}{N(p-1)}\left(\frac{\|D^s  u\|_{2}\|  u\|_{2}^{\frac{s-s_c}{s_c}}}{\|D^s  Q\|_{2}\|  Q\|_{2}^{\frac{s-s_c}{s_c}}}\right)
^{\frac{2s_c}s}\right)\\
&\geq\frac{N(p-1)-4s}{2N(p-1)}\|D^s u\|_{2}^2.
\end{align*}

\end{proof}

At the end of this section, we prove the existence result of the wave operator $\Omega^+: \phi^+\mapsto v_0$. This is important to
establish the scattering theory.

\begin{proposition}\label{wave operator}
(Existence of wave operators) Suppose that $ \phi^+\in H^s$ and that
\begin{equation}\label{4.11}
\frac{1}{2}M[\phi^+]^{\frac{s-s_c}{s_c}}\|D^s\phi^+\|_2^2<M[Q]^{\frac{s-s_c}{s_c}}E[Q].
\end{equation}
Then, there exists $v_0\in H^s$ such that $v$ globally solves \eqref{eq1} with initial data $v_0$ satisfying
$$\|D^sv(t)\|_2\|v_0\|_2^{\frac{s-s_c}{s_c}}\leq\|D^sQ\|_2\|Q\|_2^{\frac{s-s_c}{s_c}},\ \
M[v]=\|\phi^+\|_2^2,\ \  E[v]=\frac{1}{2}\|D^s\phi^+\|_2^2,$$
and
$\lim\limits_{t\rightarrow+\infty}\|v(t)-U(t)\phi^+\|_{H^s}=0$. 
Moreover, if $\|U(t)\phi^+\|_{S(\Lambda_{s_c})}\leq\delta_{sd}$, where $\delta_{sd}$ is defined in Proposition \ref{sd}, then
$$\|v\|_{S(\Lambda_{s_c})}\leq 2\|U(t)\phi^+\|_{S(\Lambda_{s_c})},\ \
         \|D^{s_c}v\|_{S(\Lambda_0)}\leq 2c\|\phi^+\|_{\dot{H}^{s_c}}.$$
\end{proposition}
\begin{proof} Let $v(t)=FNLS(t)v_0$
be the solution $v(t)$ of  the  fractional NLS \eqref{eq1} with the initial data $v(0)=v_0$.
According to the scattering theory of small initial data(see Proposition \ref{sd}),
  we consider the integral equation
\begin{equation}\label{4.12}
v(t)=U(t)\phi^+-i\int_t^\infty U(t-t^1) |v|^{p-1}v(t^1)dt^1
\end{equation}
 for $t\geq T$ with $T$ large.
From  Proposition \ref{sd},  for sufficiently large $T$, we deduce that
$
\|v\|_{S(\Lambda_{s_c};[T,\infty))}\leq2\delta_{sd},
$ and
$$\|v\|_{S(\Lambda_0;[T,\infty))}+\|D^s v\|_{S(\Lambda_0;[T,\infty))}<2c \|\phi^{+}\|_{H^s}.$$
Thus, by a similar argument when $t>T$,  we obtain
$$\|v-U(t)\phi^+\|_{S(\Lambda_0;[T,\infty))}+\|D^s( v-e^{it\Delta}
\phi^+)\|_{S(\Lambda_0;[T,\infty))} \rightarrow 0
\ \  {\rm as} \  ~T\rightarrow \infty.$$
Hence, $v(t)-U(t)\phi^+\rightarrow0 $ in $H^s,$ and thus, $M[v]=\|\phi^{+}\|_{2}^2$. By the fact
  $U(t)\phi^+\rightarrow0$ in $L^q$ for any $q\in(2,\frac{2N}{N-2s}]$ as $t\rightarrow\infty$, we have
 $\|U(t)\phi^+\|_{p+1}\rightarrow0$.  Moreover, combining this with that $\|D^s U(t)\phi^+\|_2$ is conserved,
we deduce that
 $$E[v]=\lim_{t\rightarrow\infty}(\frac{1}{2}\|D^s U(t)\phi^+\|_2^2
 -\frac{1}{p+1}\|U(t)\phi^+\|_{p+1}^{p+1})=\frac{1}{2}\|D^s \phi^+\|_2^2.$$
By the assumption  \eqref{4.11}, then we obtain $M[v]^{\frac{s-s_c}{s_c}}E[v]<E[Q]M[Q]^{\frac{s-s_c}{s_c}}$. According to
\eqref{4.11} and Remark \ref{reQ} , we deduce that
\begin{align*}
&\lim_{t\rightarrow\infty}\|v(t)\|_2^{\frac{2(s-s_c)}{s_c}}\|D^s v(t)\|_2^2
=\lim_{t\rightarrow\infty}\|U(t)\phi^+\|_2^{\frac{2(s-s_c)}{s_c}}\|D^s U(t)\phi^+\|_2^2\\
=&\|\phi^+\|_2^{\frac{2(s-s_c)}{s_c}}\|D^s \phi^+\|_2^2\leq2E[Q]M[Q]^{\frac{s-s_c}{s_c}}
=\frac{N(p-1)-4s}{N(p-1)}\|Q\|_2^{\frac{2(s-s_c)}{s_c}}\|D^s Q\|_2^2.
\end{align*}
Finally,  we can evolve $v(t)$ from $T$ back to time $0$ as in
     Theorem \ref{th3.5}.  \end{proof}

\section{Critical solution and compactness}

In the previous, we have proved the the global existence part of Theorem 1.1(see Theorem \ref{th3.5}). From now on,
  we   begin to prove  the scattering part of Theorem \ref{th1}.
 $u(t)$ is globally well-posed. According to  Proposition \ref{h1scattering}, we just need  to show that
\begin{equation}\label{Sbound}
\|u\|_{S(\Lambda_{s_c})}<\infty.
\end{equation}
Then, the $H^s$ scattering of  the solution for  Eq.\eqref{eq1} follows.

{\bf We say that $SC(u_0)$ holds if \eqref{Sbound} is true for the solution $u$ with the initial data $u_0$.}

From  Proposition \ref{sd}, we see that  there exists $\delta>0$ such that if $E[u_0]M[u_0]^{\frac{s-s_c}{s_c}}<\delta$ and
$\|u_{0}\|_{2}^{\frac{s-s_c}{s_c}}\|D^s u_{0}\|_{2}<\|Q\|_{2}^{\frac{s-s_c}{s_c}}\|D^s Q\|_{2},$
 then \eqref{Sbound} holds.
Now, for each $\delta$, we define the set $S_\delta$ to be the collection of all such initial data in $H^s$ :
$$S_\delta=\{u_0\in H^s:\ \  E[u_0]M[u_0]^{\frac{s-s_c}{s_c}}<\delta \ \
{\rm and} \ \  M[u_{0}]^{\frac{s-s_c}{s_c}}\|D^s u_{0}\|_{2}^2<M[Q]^{\frac{s-s_c}{s_c}}\|D^s Q\|_{2}^2\}.$$
We also define that $(ME)_c=\sup\{\delta:\ \ u_0\in S_\delta\Rightarrow SC(u_0)\ \  holds \}.$
If $(ME)_c=M[Q]^{\frac{s-s_c}{s_c}}E[Q]$, then we are done. Thus, we assume that
\begin{equation}\label{me}
(ME)_c<M[Q]^{\frac{s-s_c}{s_c}}E[Q].
\end{equation}
Then, there exist   solutions $u_n$  of \eqref{eq1} with $H^s$ initial data $u_{n,0}$
(after rescaling, we may inquire that $u_n$ satisfies $\|u_n\|_2=1$ )
such that $\|D^s u_{n,0}\|_{2}<\|Q\|_{2}^{\frac{s-s_c}{s_c}}\|D^s Q\|_{2}$  and $E[u_{n,0}]\downarrow (ME)_c$ as
$n\rightarrow \infty,$
and   $SC(u_0)$ does not hold for any $n$.

In this section, we will prove that there exists a critical  $H^s$ solution $u_c$ to \eqref{eq1} with initial data
$u_{c,0}$ such that $\| u_{c,0}\|_{2}^{\frac{s-s_c}{s_c}}\|D^s u_{c,0}\|_{2}<\|Q\|_{2}^{\frac{s-s_c}{s_c}}\|D^s Q\|_{2}$
and $M[u_c]^{\frac{s-s_c}{s_c}}E[u_c]= (ME)_c$
for which $SC(u_{c,0})$ does not hold.  Then, we
will show that the set
$\{u_c(\cdot,t)|0\leq t<+\infty\}$ is precompact in  $H^s$. Finally, we will use these properties to obtain  the rigidity theorem in Section 5, which will use to conduct a contradiction. This can be used to finish the proof of Theorem 1.1.

First, we will  introduce a  profile decomposition lemma that is highly similar to that in \cite{radial}, which were firstly proposed by Keraani \cite{keraani2001} for
 for the cubic Schr\"{o}dinger equation, as follows.

\begin{lemma}\label{lpd}
(Profile expansion) Let $\phi_{n}(x)$ be a radial and uniformly bounded
sequence in $H^s$.
Then, for each M, there exists a subsequence of $\phi_{n}$, also
denoted by $\phi_{n}$, and \\
(1) for each $1\leq j\leq M$, there
exists a (fixed in n) profile $\psi^{j}(x)$ in $H^s$,\\
 (2) for each $1\leq j\leq M$, there exists a sequence (in n) of time
shifts $t_{n}^{j}$, \\
(3) there exists a
sequence (in n) of remainders $W_{n}^{M}(x)$ in $H^s$ such that
$$\phi_{n}(x)=\sum_{j=1}^{M}U(-t_{n}^{j})\psi^{j}(x)+W_{n}^{M}(x).$$
The time and space sequences have a pairwise divergence property,
i.e., for $1\leq j\neq k\leq M$, we have
\begin{equation}\label{divergence}
\lim_{n\rightarrow+\infty}
|t_{n}^{j}-t_{n}^{k}|=+\infty.
\end{equation}
The remainder sequence has the following asymptotic smallness
property:
\begin{equation}\label{remainder}
\lim_{M\rightarrow+\infty}[\lim_{n\rightarrow+\infty}\|U(t)W_{n}^{M}\|_{S(\Lambda_{s_c})}]=0.
\end{equation}
For fixed M and any $0\leq \alpha\leq s$, we have the asymptotic
Pythagorean expansion:
\begin{equation}\label{hsexpansion}
\|\phi_{n}\|_{\dot{H}^{\alpha}}^{2}=\sum_{j=1}^{M}\|\psi^{j}\|_{\dot{H}^{\alpha}}^{2}+\|W_{n}^{M}\|_{\dot{H}^{\alpha}}^{2}+o_{n}(1).
\end{equation}
\end{lemma}

\begin{proof}\label{reproof}
The proof of the above linear profile decomposition for the  fractional NLS is quite similar with that for the  fourth-order nonlinear Schr\"{o}dinger equation in \cite{qgcpde16}. Here, we omit the main proof. But we should point out that
  \eqref{remainder} could  be improved to
\begin{equation}\label{remainder'}
\lim_{M\rightarrow+\infty}[\lim_{n\rightarrow+\infty}\|U(t)W_{n}^{M}\|_{L^qL^r}]=0,\ \ \forall (q,r)\ {\rm satisfies}\ \eqref{gap}\
{\rm with}\ \theta=s_c.
\end{equation}
More precisely,
\begin{equation}\label{remainder''}
\lim_{M\rightarrow+\infty}[\lim_{n\rightarrow+\infty}\|U(t)W_{n}^{M}\|_{L^\infty L^{\frac{2N}{N-2s_c}}}]=0.
\end{equation}
\end{proof}

Using the profile expansion, similar argument as in \cite{qgcpde16} could just be applied to obtain the following   results: Energy expansion,  Existence of a critical solution
and Precompactness of the flow of the critical solution.
Note that we have also proved similar counterparts of these results with respect to the fractional Hartree equation \cite{g-zarxiv1}, and we omit the proof here.
 \begin{lemma}\label{energy  expansion}(Energy Pythagorean expansion)
 In the situation of Lemma \ref{lpd}, we have
 \begin{equation}\label{epe}
E[\phi_{n}]=\sum_{j=1}^{M}E[U(-t_n^j)\psi^{j}]+E[W_{n}^{M}]+o_n(1).
\end{equation}
\end{lemma}

\begin{proposition}\label{critical}
(Existence of a critical solution)
There exists a global solution $u_c$ in $H^s$ with initial data $u_{c,0}$ such that
$\|u_{c,0}\|_2=1,$
$$E[u_c]=(ME)_c<M[Q]^{\frac{s-s_c}{s_c}}E[Q],\ \ \ \|D^s u_c\|_2^2<M[Q]^{\frac{s-s_c}{s_c}}\|D^s Q\|_2^2,\ \ for \ \ all\ \ 0\leq t<\infty,$$
and $$\| u_c\|_{S(\Lambda_{s_c})}=+\infty.$$
\end{proposition}

\begin{proposition}\label{procompact}
(Precompactness of the flow of the critical solution)
Let $u_c$ be as  in Proposition \ref{critical}; then, if $\|u_c\|_{S([0,+\infty);\Lambda_{s_c})}=\infty$,
$$\{u_c(\cdot,t)| ~t\in[0,+\infty)\}\subset H^s$$
 is precompact in $H^s$.  A corresponding conclusion is reached if
 $\|u_c\|_{S((-\infty,0];\Lambda_{s_c})}=\infty$.

\end{proposition}

\begin{corollary}\label{compact-localization}
Let $u=u(t)$ be a solution to \eqref{eq1} such that $\mathcal{K}^+=\{u(\cdot,t)| ~t\in[0,+\infty)\}$
is precompact in $H_r^s$. Then, for each $\epsilon>0,$
there exists $R>0$ such that
$$\int_{|x|>R}|D^s u(x,t)|^2+|u(x,t)|^2+|u(x,t)|^{p+1}dx\leq\epsilon.$$
\end{corollary}
\begin{proof}
If not, for any $R>0$, there exists $\epsilon_0>0$ and a sequence $t_n$ such that
$$\int_{|x|>R}|D^s u(x,t_n)|^2+|u(x,t_n)|^2+|u(x,t_n)|^{p+1}dx\geq\epsilon_0.$$
By the precompactness of $\mathcal{K}^+$, there exists $\phi\in H^s$ such that, up to a subsequence of $t_n$,
we have $u(\cdot,t_n)\rightarrow\phi$ in $H^s$. Thus, for any $R>0$, we obtain
$$\int_{|x|>R}|D^s \phi(x)|^2+|\phi(x)|^2+|\phi(x)|^{p+1}dx\geq\epsilon_0,$$
from which we can easily obtain a contradiction because  $\phi\in H^s$ and
$\|\phi\|_{p+1}^{p+1}\leq c\|\phi\|^{p+1}_{H^s}$ by the Sobolev inequality.
\end{proof}

\section{Rigidity theorem }

In order to finish the proof of Theorem 1.1, we need  the following rigidity theorem.

\begin{theorem}\label{rigidity}  Let $N\geq 2$ and  $1+\frac{4s}{N}<p<1+\frac{4s}{N-2s}$.
Assume that the initial data $u_0\in H^s$ is radial and   $u_0\in K_g$, i.e.,
\begin{equation}\label{6.1}
M[u_{0}]^{\frac{s-s_c}{s_c}}E[u_{0}]<M[Q]^{\frac{s-s_c}{s_c}}E[Q],
\end{equation}
 and
\begin{equation}\label{6.2}
M[u_{0}]^{\frac{s-s_c}{s_c}}\|u_{0}\|^2_{\dot{H}^s}<M[Q]^{\frac{s-s_c}{s_c}}\| Q\|^2_{\dot{H}^s}.
\end{equation}
Let $u=u(t)$ be the riadially global solution of \eqref{eq1} with initial data $u_0$.  If it holds that
$\mathcal{K}^+=\{u(\cdot,t)| ~t\in[0,+\infty)\}$ is precompact in $H^s$,
then $u_0=0$.
The same conclusion holds if
$\mathcal{K}^-=\{u(\cdot,t):t\in(-\infty,0]\}$ is precompact in $H^s$.
\end{theorem}

Now, we introduce the localized virial estimate for the radial solutions of \eqref{eq1} in terms of
the idea in \cite{b-h-ljfa16}.
Let  $u\in H^s$ with $s\geq\frac12$. we define the auxiliary function $u_m=u_m(t,x)$ as
\begin{align}\label{um}
u_m:=c_s\frac1{-\Delta+m}u(t)=c_s\mathcal F^{-1}\frac{\widehat u(t,\xi)}{|\xi|^2+m}
\end{align}
where  $c_s=\sqrt{\frac{sin\pi s}{\pi}}$.
It follows from  \cite{b-h-ljfa16} that, for any  $u\in H^s$,
\begin{align}\label{snorm}
\int_0^\infty m^s\int_{\mathbb R^N}|\nabla u_m|^2dxdm=s\|(-\Delta)^{\frac s2}u\|_2^2.
\end{align}
Then, we obtain the following corollary, which is a counterpart of Corollary \ref{compact-localization}.
\begin{corollary}\label{corlocal}
Let $u=u(t)$ be a solution to \eqref{eq1} such that $K=\{u(\cdot,t)| ~t\in[0,+\infty)\}$
is precompact in $H_r^s$. Then, for each $\epsilon>0,$
there exists $R>0$ such that
$$\int_0^\infty m^s\int_{|x|>R}|\nabla u_m|^2dxdm+\int_{|x|>R}|u(x,t)|^2+|u(x,t)|^{p+1}dx\leq\epsilon.$$
\end{corollary}

{\bf The Proof of Theorem \ref{rigidity}.}
It suffices to address the $\mathcal{K}^+$ case, since the $\mathcal{K}^-$ case follows similarly.

Let $\varphi\in C_c^\infty$ be a radially real-valued function, and defined by
$$\varphi(x)=\left\{\begin{aligned}
&|x|^2 \ \ &{\rm for} \ \  |x|\leq1\\
&0\ \ &{\rm for} \ \ |x|\geq2.\end{aligned}\right.$$
For any $R>0$,  take $\varphi_R(x)=R^2\varphi(\frac xR)$
and define the localized virial of $u\in H^s$  by
\begin{align*}
\mathcal J_{R}(t):=2Im\int_{\mathbb R^N}\bar u(t,x)\nabla \varphi_R( x)\cdot\nabla u(t,x)dx.
\end{align*}
Similar to  the method  in \cite{b-h-ljfa16}, we obtain  the identity
\begin{align*}
\mathcal J'_{R}(t))&=\int_0^\infty m^s\int_{\mathbb R^N}
\left(4\overline{\partial_k u_m}(\partial^2_{kl}\varphi(\frac xR))\partial_lu_m-(\Delta^2\varphi_R(x))|u_m|^2
\right)dxdm\\
&\qquad \qquad -\frac{2(p-1)}{p+1}\int_{\mathbb R^N}(\Delta\varphi)(\frac xR)|u|^{p+1}dx.
\end{align*}
By  the properties  of $\varphi$,  we deduce that
\begin{align}\label{MR}
\mathcal J'_R(t)\geq&8\int_0^\infty m^s\int_{|x|\leq R}|\nabla u_m|^2dx+
4\int_0^\infty m^s\int_{R<|x|<2R}\partial^2_r\varphi\left(\frac{x}{R}\right)|\nabla u_m|^2dxdm
\\ \nonumber
&-\int_0^\infty m^s\int_{|x|>R}\Delta^2\varphi_R\left(x\right)|u_m|^2dxdm\\ \nonumber
&-\frac{2(p-1)}{p+1}\int_{|x|\leq R}|u|^{p+1}dx-c\int_{R<|x|<2R}|u|^{p+1}dx\\ \nonumber
\geq&\left(8\int_0^\infty m^s\int_{\mathbb R^N}|\nabla u_m|^2dx-\frac{4N(p-1)}{P+1}\int|u|^{p+1}dx\right)\\ \nonumber &-
\left(8\int_0^\infty m^s\int_{|x|>R}|\nabla u_m|^2dx-\frac{4N(p-1)}{P+1}\int_{|x|>R}|u|^{p+1}dx\right)\\ \nonumber
&-c\left(\int_0^\infty m^s\int_{R<|x|<2R}|\nabla u_m|^2dx+\int_{R<|x|<2R}|u|^{p+1}dx\right)-\frac c{R^{2s}}\|u\|_2^2.
\end{align}
Here, we use the following estimate in the last step(see \cite{b-h-ljfa16}),
\begin{align*}
\int_0^\infty m^s\int_{|x|>R}\Delta^2\varphi_R\left(x\right)|u_m|^2dxdm\leq cR^{-2s}\|u\|_2^2.
\end{align*}
Now, let $\delta\in(0,1)$ satisfy
$E[u_0]<(1-\delta)E[Q]M[Q]^{\frac{s-s_c}{s_c}}$.
From
Lemma \ref{lemlowerbound} and Lemma \ref{lemcompare}, we see that
\begin{align*}
8\int_0^\infty m^s\int_{\mathbb R^N}|\nabla u_m|^2dx-\frac{4N(p-1)}{P+1}\int|u|^{p+1}dx&=8s\gamma\|D^s u\|_2^2-\frac{4N(p-1)}{P+1}\int|u|^{p+1}dx\\
&\geq C_\delta\|D^s u_0\|_2^2.
\end{align*}
Then, we can take  $R$ large enough to obtain the following estimate
\begin{align}\label{lowerM'}
\mathcal J'_R(t)\geq&C\|D^s u_0\|_2^2.
\end{align}
Integrate \eqref{lowerM'} over $[0,t]$.
$$|\mathcal J_R(t)-\mathcal J_R(0)|\geq C t\|D^s u_0\|_2^2$$
However, by \cite{b-h-ljfa16}, we should
have
$$|\mathcal J_R(t)-\mathcal J_R(0)|\leq C_R(\|u\|^2_{H^{\frac12}}+\|u_0\|^2_{H^{\frac12}})
\leq C_R(\|u\|^2_{H^{s}}+\|u_0\|^2_{H^{s}})\leq C_R\|Q\|^2_{H^{s}},$$
which is a contradiction for large $t$ unless $u_0=0$.
\ \  \ \ \ \ \ \ \ \ \ \ \  \ \ \ \ \ \ \ \ \ \ \  \ \ \ \ \ \ \ \ \ \ \  \ \ \ \ \ \ \ \ \ \ \  \ \ \ \ \ \ \ \ \ \ \ $\Box$
\vspace{0.5cm}

Now, we can finish the proof of Theorem 1.1.\\
{\bf The  Proof of Theorem 1.1.}

Note that by Proposition \ref{procompact},
the critical solution $u_c$ constructed in Section 4 satisfies the hypotheses in Theorem \ref{rigidity}.
Therefore,
to complete the proof of Theorem \ref{th1}, we should apply Theorem \ref{rigidity} to
$u_c$ and find that $u_{c,0}=0$, which contradicts the fact that $\|u_c\|_{S(\Lambda_{s_c})}=+\infty.$
This contradiction  shows that $SC(u_0)$ holds. Thus, by Proposition \ref{h1scattering},  we have shown that $H^s$ scattering holds.
\ \  \ \ \ \ \ \ \ \ \ \ \  \ \ \ \ \ \ \ \ \ \ \  \ \ \ \ \ \ \ \ \ \ \  \ \ \ \ \ \ \ \ \ \ \  \ \ \ \ \ \ \ \ \ \ \ $\Box$

\vspace{0.5cm}

{\bf Acknowledgments }

This work  was  supported by the National Natural Science Foundation of China (No. 11501395,  11301564, and 11371267) and  the Excellent Youth Foundation of Sichuan  Scientific Committee grant No. 2014JQ0039  in China.

\vspace{0.5cm}


\end{document}